\documentclass[reqno,11pt,centertags]{amsart}
\usepackage{amssymb}
\usepackage{mathptm}

\usepackage[english]{babel}

\DeclareSymbolFont{SY}{U}{psy}{m}{n}
\DeclareMathSymbol{\emptyset}{\mathord}{SY}{'306}

\newcommand{\slv}{{\textsl{v}}}
\renewcommand{\eqref}[1]{{\rm(\ref{#1})}}

\newcommand{\bbR}{{\mathbb R}}

\newcommand{\cB}{{\mathcal B}}

\newcommand{\cG}{{\mathcal G}}

\newcommand{\conv}{\mathop{\rm conv}}

\newcommand{\rl}{{\mathit{\,\,l}}}
\newcommand{\rr}{{\mathit{\,r}}}

\newcommand{\sE}{{\sf E}}

\newcommand{\fA}{\mathfrak{A}}

\newcommand{\fH}{\mathfrak{H}}
\newcommand{\fK}{\mathfrak{K}}
\newcommand{\fL}{\mathfrak{L}}
\newcommand{\fM}{\mathfrak{M}}
\newcommand{\fN}{\mathfrak{N}}
\newcommand{\fP}{\mathfrak{P}}
\newcommand{\fQ}{\mathfrak{Q}}

\newcommand{\dist}{\mathop{\rm dist}}

\newcommand{\lal}{{\langle}}
\newcommand{\ral}{{\rangle}}

\newcommand{\be}{\begin{equation}}
\newcommand{\ee}{\end{equation}}

 \DeclareMathOperator{\spec}{spec}

\newcommand{\Ran}{\mathop{\mathrm{Ran}}}

\newcommand{\Dom}{\mathop{\mathrm{Dom}}}
\newcommand{\Ker}{\mathop{\mathrm{Ker}}}

\allowdisplaybreaks

\numberwithin{equation}{section}

\newtheorem{introtheorem}{Theorem}
\newtheorem{theorem}{Theorem}[section]

\newtheorem{lemma}[theorem]{Lemma}

\newtheorem{hypothesis}[theorem]{Hypothesis}
\theoremstyle{definition}

\theoremstyle{remark}
{\it}{\rm}
\newtheorem{remark}[theorem]{Remark}

\begin{document}

\title
[An alternative proof of the a priori $\textrm{TAN}\,\,\Theta$
Theorem] {An alternative proof of the a priori $\textrm{\bf
TAN}\,\boldmath{\Theta}$ Theorem}

\author
[A.\,K.\,Motovilov] {Alexander K. Motovilov}

\address{A.K.Motovilov, Bogoliubov Laboratory of Theoretical Physics, JINR\\
Joliot-Curie 6, 141980 Dubna, Moscow Region, Russia}
\email{motovilv@theor.jinr.ru}


\keywords{Perturbation problem for spectral subspaces,
operator Riccati equation, $\tan\theta$ theorem}

\begin{abstract}
\small

Let $A$ be a self-adjoint operator in a separable Hilbert space.
Suppose that the spectrum of $A$ is formed of two isolated
components $\sigma_0$ and $\sigma_1$ such that the set $\sigma_0$
lies in a finite gap of the set $\sigma_1$. Assume that $V$ is a
bounded additive self-adjoint perturbation of $A$, off-diagonal with respect
to the partition \mbox{$\spec(A)=\sigma_0\cup\sigma_1$}. It is known
that if \mbox{$\|V\|<\sqrt{2}\,\,\dist(\sigma_0,\sigma_1)$}, then
the spectrum of the perturbed operator $L=A+V$ consists of two
disjoint parts $\omega_0$ and $\omega_1$ which originate from the
corresponding initial spectral subsets $\sigma_0$ and $\sigma_1$.
Moreover, for the difference of the spectral projections
$\sE_A(\sigma_0)$ and $\sE_{L}(\omega_0)$ of $A$ and $L$ associated
with the spectral sets $\sigma_0$ and $\omega_0$, respectively, the
following sharp norm bound holds:
$$
\|\sE_A(\sigma_0)-\sE_{L}(\omega_0)\| \leq
\sin\left(\arctan\frac{\|V\|}{\dist(\sigma_0,\sigma_1)}\right).
$$
In the present note, we give a new proof of this bound for
\mbox{$\|V\|<\dist(\sigma_0,\sigma_1)$}.

\end{abstract}

\maketitle

\bigskip

\hfill\textit{Dedicated to the memory of my university advisor
Stanislav Petrovich Merkuriev}

\bigskip

\bigskip

\section{Introduction}
\label{SIntro}

One of fundamental problems in perturbation theory of linear
operators consists in the study of variation of invariant and, in
particular, spectral subspaces under an additive perturbation  (see,
e.g., \cite{BDM1983,DK70,K}). In perturbation theory of self-adjoint
operators, the classical trigonometric estimates on the rotation of
spectral subspaces have been established by Davis and Kahan
\cite{DK70}. As regards the history of the subject and surveys on
other known subspace variation bounds in the self-adjoint
perturbation problem, we refer the reader, e.g., to
\cite{AM2014,AM2013,SeelPhD}.

The author's interest to the subspace perturbation theory arose due
to a series of works performed with a very immediate participation
of S.P.Merkuriev and devoted to the construction of three-body
Hamiltoninas with pairwise interactions depending on the energies of
two-body subsystems (in particular, see papers
\cite{YaF1988,JMP1990} and references therein). The attempts
\cite{M-JMP1995,MotJMPh91} to answer the question
\cite{McKellarMcCay} on the possibility to replace, in the two-body
Schr\"odinger equation, a pairwise energy-dependent potential by a
spectrally  equivalent (i.e. keeping the original spectrum and
original eigenfunctions) energy-independent potential led the author
to a study of the solvability of an operator Riccati equation.
Actually, the mentioned question on the replacement of the
energy-dependent potential by an equivalent energy-independent one
is related to a more general question on the Markus-Matsaev-type
factorization of operator-valued functions of the complex variable
and the existence of operator roots for those functions (see
\cite{MenShk,MrMt,LMMT}). The Schr\"odinger operator with a
potential depending on the energy like a resolvent represents an
example of the operator-valued function, the construction of
operator roots for which necessarily reduces to the search for graph
representations of invariant subspaces for some $2\times 2$ block
operator matrix  (see
\cite{M-JMP1995,MenShk,LMMT,Tretter-Book,AMM}). In its turn,
answering the question on the existence of the graph representations
leads to the study of mutual geometry of unperturbed and perturbed
invariant subspaces and, among other things, to the study of
spectral properties of the operator angles between these subspaces (see
\cite{KMM2} and references cited therein).

In the present work we consider a particular case of the
self-adjoint subspace perturbation problem. Namely, we suppose that
the spectrum $\spec(A)$ of the initial self-adjoint operator $A$,
acting in a separable Hilbert space $\fH$, consists of a two
disjoint parts $\sigma_0$ and $\sigma_1$ such that the
first of them lies in a finite gap of the second one. (We
recall that by a finite gap of a closed set $\sigma\subset\bbR$ one
understands an open bounded interval on $\bbR$ that does not
intersect $\sigma$ but both of its ends belong to $\sigma$.) For
future references, we write our assumption on the mutual
position of the spectral components $\sigma_0$ and $\sigma_1$ in the
following form:
\begin{equation}
\label{tant} \spec(A)=\sigma_0\cup\sigma_1, \quad
\overline{\conv(\sigma_0)}\cap\overline{\sigma}_1=\emptyset
\quad\text{and}\quad \sigma_0\subset\conv(\sigma_1),
\end{equation}
where the symbol $\mathop{\mathrm{conv}}$ denotes the convex hull,
and the overlining means closure.

As for the (additive) perturbation $V$, we assume that it is a
bounded self-adjoint operator on $\fH$. Furthermore, it is supposed
that this perturbation is off-diagonal with respect to the partition
$\spec(A)=\sigma_0\cup\sigma_1$, which means that $V$ anticommutes
with the difference $\sE_A(\sigma_{0})-\sE_A(\sigma_{1})$ of the
spectral projections $\sE_A(\sigma_{0})$ and $\sE_A(\sigma_{1})$ of
$A$ corresponding to the spectral sets $\sigma_0$ and
$\sigma_1$.

Let  $d:=\dist(\sigma_0,\sigma_1)$ be the distance between the
sets $\sigma_0$ and $\sigma_1$. For the spectral
disposition \eqref{tant}, it has been established in  \cite{KMM3}
(also see \cite{Tretter-Book}) that the gaps between $\sigma_0$ and
$\sigma_1$ do not close under an off-diagonal self-adjoint
perturbation $V$ if its norm satisfies the condition
\begin{equation}
\label{V2d} \| V \| < \sqrt{2}\,\,d.
\end{equation}
Moreover, this condition is optimal. If \eqref{V2d} holds, the
spectrum of the perturbed operator $L=A+V$ is represented by the
union of two isolated  sets $\omega_0\subset\Delta$ and
$\omega_1\subset\bbR\setminus\Delta$ where $\Delta$ stands just for
the very same finite gap of the spectral set $\sigma_1$ that contains the
whole complementary spectral set~$\sigma_0$.

For the spectral disposition \eqref{tant} and off-diagonal
self-adjoint perturbations $V$, the optimal bound on the variation
of the spectral subspaces of a self-adjoint operator  has been
established, step by step, in papers \cite{MotSel} (for the case
where $\|V\|<d$) and \cite{AM2012} (for $d\leq\|V\|<\sqrt{2}d$); it
is written in terms of the difference of the respective spectral
projections $\sE_{A}(\sigma_0)$ and $\sE_{L}(\omega_0)$ of the
unperturbed and perturbed operators $A$ and $L$. We reproduce this
bound in the following statement (cf. \cite[Theorem 2]{MotSel} and
\cite[Theorem 1]{AM2012}).

\begin{introtheorem}
\label{Th1} Let $A$ be a self-adjoint operator on a separable
Hilbert space $\fH$, and let the spectrum of $A$ consist of a two
isolated components $\sigma_0$ and $\sigma_1$ which satisfy the
condition \eqref{tant}. Assume that $V$ is a bounded self-adjoint
operator on $\fH$, off-diagonal with respect to the partition
$\spec(A)=\sigma_0\cup\sigma_1$, and set $L=A+V$, $\Dom(L)=\Dom(A)$.
Furthermore, assume that the norm of $V$ satisfies inequality
\eqref{V2d}, and let $\omega_0=\spec(L)\cap\Delta$. Then the
following norm bound holds:
\begin{equation}
\label{edif1} \|\sE_{A}(\sigma_0)- \sE_{L}(\omega_0)\| \leq \sin
\left( \arctan\frac{\|V\|}{d}\right).
\end{equation}
\end{introtheorem}

Recall that if $P$ and $Q$ are orthogonal projections in a Hilbert
space then the quantity
\begin{equation}
\label{TetPQ} \Theta(\fP,\fQ):=\arcsin|P-Q|,
\end{equation}
where $|P-Q|=\sqrt{(P-Q)^2}$ stands for the absolute value of $P-Q$,
is called the \emph{operator angle} between the subspaces
$\fP=\Ran(P)$ and $\fQ=\Ran(Q)$. A substantial discussion of the
term {\guillemotleft}operator angle{\guillemotright} and relevant
references may be found, in particular, in the recent paper
\cite[Section 2]{SeelNotes}. In its turn, the norm of the operator
angle $\Theta(\fP,\fQ)$ determines the \emph{maximal angle}
$\theta(\fP,\fQ)$ between $\fP$ и $\fQ$ (see \cite{AM2013}). Namely,
$\theta(\fP,\fQ)=\arcsin\|P-Q\|$.

In view of \eqref{TetPQ} the estimate \eqref{edif1} is equivalent to
the bound
\begin{equation}
\label{Tet} \tan\Theta(\fA_0,\fL_0)\leq\frac{\|V\|}{d}
\end{equation}
for the tangent of the operator angle $\Theta(\fA_0,\fL_0)$ between
unperturbed and perturbed spectral subspaces
$\fA_0=\Ran\bigl(\sE_A(\sigma_0)\bigr)$ and
$\fL_0=\Ran\bigl(\sE_L(\omega_0)\bigr)$, respectively. Unlike the
Davis-Kahan $\tan\Theta$ theorem \cite[Theorem 6.3]{DK70} and its
extensions in \cite{KMM3} and \cite{Naka}, the bound \eqref{edif1}
involves the distance only between the unperturbed spectral
components $\sigma_0$ and $\sigma_1$. This is why it is called in
\cite{MotSel} and \cite{AM2012} the \emph{a priori} $\tan\Theta$
theorem.

The proof of Theorem \ref{Th1} for $\|V\|<d$ given in \cite{MotSel},
is based on the study of location of the spectrum of the product
$J'J$ of the self-adjoint involutions
$J=\sE_A(\sigma_0)-\sE_A(\sigma_1)$ and
$J'=\sE_L(\omega_0)-\sE_L(\omega_1)$. In this way, the proof in
\cite{MotSel} requires a reformulation of the subspace perturbation
problem into the language of pairs of involutions. It
also relies on a knowledge of certain properties of the polar
decomposition for maximal accretive operators.

The present note is aimed at giving a proof of \eqref{edif1} that
would be independent from the approach suggested in \cite{MotSel}.
Moreover, we think that the proof presented below in Section
\ref{Sec3} is simpler and more straightforward than
the proof in \cite{MotSel}. Our new proof is based on the rather
standard technique  \cite{AMM} involving the reduction of the
invariant subspace perturbation problem under consideration to the
study of the operator Riccati equation
\begin{equation}
\label{Ric0} XA_0-A_1X+XBX=B^*,
\end{equation}
where $A_0=A\bigr|_{\fA_0}$ and $A_1=A\bigr|_{\fA_1}$ are the parts
of the self-adjoint operator $A$ in its spectral subspaces
$\fA_0=\Ran\bigl(\sE_A(\sigma_0)\bigr)$ and
$\fA_1=\Ran\bigl(\sE_A(\sigma_1)\bigr)$, and $B=V\bigr|_{\fA_1}$. As
it was established in \cite{KMM3}, the perturbed spectral subspace
$\fL_0=\Ran\bigl(\sE_L(\omega_0)\bigr)$ is the graph of a solution
$X\in\cB(\fA_0,\fA_1)$ to equation \eqref{Ric0}. The latter means
(see, e.g., \cite{KMM2}) that
\begin{equation}
\label{EEX} \|\sE_{A}(\sigma_0)-\sE_{L}(\omega_0)\|=\sin \left(
\arctan\|X\|\right).
\end{equation}
Thus, when one obtains a bound for the norm of the solution $X$,
one simultaneously  finds an estimate for the norm of the
difference of the spectral projections $\sE_{L}(\omega_0)$ and
$\sE_{A}(\sigma_0)$.

In our derivations we rely on the identities established in
\cite[Lemma 2.2]{AM2012} for eigenvalues and eigenvectors of the
absolute value $|X|=\sqrt{X^*X}$ of the operator $X$. Starting from
these identities (more precisely, from the identities \eqref{Id2}
and \eqref{Id1} below), we obtain,  for $\|\sE_{A}(\sigma_0)-
\sE_{L}(\omega_0)\|$, a bound \eqref{thebest} which is stronger but
more detail than the estimate \eqref{edif1}: The bound
\eqref{thebest} involves, along with $\|V\|$ and $d$, also the
length $|\Delta|$ of the gap $\Delta$. The estimate \eqref{thebest}
reproduces the main result of \cite[Theorem 5.3]{MotSel}. Just the
proof of this estimate, alternative to the proof in \cite{MotSel},
we consider as the principal result of the present work. It was
already pointed out in \cite{MotSel} that the bound \eqref{edif1} is
nothing but a simple corollary to the more detail estimate
\eqref{thebest}.

Let us describe the structure of the paper. Alongside with the
already mentioned identities \eqref{Id2} and \eqref{Id1} for
eigenvalues and eigenvectors of the absolute value of the solution
to the Riccati equation \eqref{Ric0}, Section \ref{SecOR} contains a
selection of known results concerning the location of the spectrum
of the perturbed operator $L=A+V$ and the solvability of \eqref{Ric0}
under the weaker than \eqref{V2d} but more detail condition
$\|V\|<\sqrt{d|\Delta|}$. The principal result of the work --- the
new proof of the bound on rotation of the spectral subspace
$\fA_0=\Ran\bigl(\sE_{A}(\sigma_0)\bigr)$ under an off-diagonal
perturbation $V$ satisfying condition
$\|V\|<\sqrt{d(|\Delta|-d)}$\,\, --- is presented in Section
\ref{Sec3} (see Theorem \ref{T-MS}). Section \ref{Sec3} comes to the
end with the proof of Theorem \ref{Th1}.

Throughout the paper, by a subspace we always understand a closed
linear subset of a Hilbert space. The notation $\cB(\fM,\fN)$ is
applied to the Banach space of bounded linear operators from a
Hilbert space $\fM$ to a Hilbert space в $\fN$. The orthogonal sum
of two Hilbert spaces (or orthogonal subspaces) $\fM$ and $\fN$ is
denoted by $\fM\oplus\fN$.  The graph
$\cG(K):=\{y\in\fM\oplus\fN\,|\,\,y=x\oplus Kx,\,\, x\in\fM\}$ of an
operator $K\in\cB(\fM,\fN)$ is called the graph subspace associated
with $K$. At the same time, the operator $K$ itself is called the
angular operator associated with the ordered pair
$\bigl(\fM,\fK\bigr)$ of the subspaces $\fM$ and $\fK=\cG(K)$. The
notations $\Dom(Z)$ and $\Ran(Z)$ are used respectively for the
domain and the range of a linear operator~$Z$.

\section{Preliminaries}
\label{SecOR}

It is convenient for us to use a block matrix representation of the
operators under consideration. Thus, we adopt the following
hypothesis (notice that this hypothesis does not yet concern the
mutual position of the spectra of the operators $A_0$ and $A_1$).

\begin{hypothesis}
\label{HypL} Let $\fA_0$ and $\fA_1$ be complementary orthogonal
subspaces of a separable Hilbert space $\fH$, i.e.
$\fH=\fA_0\oplus\fA_1$. Assume that $A$ is a self-adjoint operator
in $\fH$ admitting the block diagonal matrix representation
\begin{align}
\label{Adiag}
A & = \left(\begin{array}{cc} A_0 & 0\\
0 & A_1
\end{array}\right), \quad \Dom(A)=\fA_0\oplus\Dom(A_1),
\end{align}
where $A_0$ is a bounded self-adjoint operator on $\fA_0$, and
$A_1$, a possibly unbounded self-adjoint operator in $\fA_1$. Let
$V$ be a bounded self-adjoint operator on $\fH$, off-diagonal with
respect to the decomposition $\fH=\fA_0\oplus\fA_1$, that is,
\begin{align}
\label{Voff}
V & =\left(\begin{array}{cc} 0 & B\\
B^* & 0
\end{array}\right),
\end{align}
where $0\neq B\in\cB(\fA_1,\fA_0)$. Assume that $L=A+V$,
$\Dom(L)=\Dom(A)$, and, hence,
\begin{equation}
\label{Ltot} L=\begin{pmatrix} A_0 & B \\ B^* & A_1 \end{pmatrix},
\quad \Dom(L)=\fA_0\oplus\Dom(A_1).
\end{equation}
\end{hypothesis}

Under Hypothesis \ref{HypL}, an operator
$X\in\cB(\fA_0,\fA_1)$ is said to be a solution to the operator
Riccati equation \eqref{Ric0} if
\begin{equation}
\label{DefSolRic} \Ran(X)\subset\Dom(A_1)
\end{equation}
and equality \eqref{Ric0} holds as an operator identity (cf., e.g.,
\cite[Definition 3.1]{AMM}). Clearly, the solution $X$ (if it
exists) may only be non-zero. Otherwise, $X=0$ would imply
that $B=0$, which contradicts the assumption.
In what follows, by $U$ we denote the partial isometry in the polar
decomposition $X=U|X|$ of the solution $X$. In particular, for $U$
we have
\begin{equation}
\label{Kisom} U \text{ \ is an isometry on\, } \Ran(|X|)=\Ran(X^*).
\end{equation}
We adopt the convention that the
action of $U$ is trivially extended onto the kernel
$\Ker(X)=\Ker(|X|)$ of $X$, i.e.
\begin{equation*}
U|_{\Ker(X)}=0.
\end{equation*}
In such a case the operator  $U$ is uniquely defined on all the
subspace $\fA_0$ (see, e.g., \cite[Theorem 8.1.2]{BirSol}).

We will need two identities for eigenvalues and eigenvectors (in
case if they exist) of the absolute value  $|X|$. The first of these
identities has been established in \cite[Lemma 2.2]{AM2012}. The
second identity trivially follows from other identities also proven
in \cite[Lemma 2.2]{AM2012}.

\begin{lemma}[\cite{AM2012}]
\label{Lcr} Assume Hypothesis \ref{HypL}. Suppose that the Riccati
equation \eqref{Ric0} possesses a solution $X\in\cB(\fA_0,\fA_1)$
and that the absolute value $|X|$ of this solution has an eigenvalue
$\lambda\geq 0$. Let $u$, $u\neq 0$, be an eigenvector of $|X|$
corresponding to the eigenvalue $\lambda$, i.e.  $|X|u=\lambda u$.
If $U$ is the isometry from the polar decomposition $X=U|X|$ of $X$,
then $Uu\in\Dom(A_1)$ and the following two identities hold:
\begin{align}
\nonumber
\lambda\bigl(\|A_1Uu\|^2+\|BUu\|^2-\|A_0u\|^2-&\|B^*u\|^2\bigr)\\
\label{Id2} =-(1-\lambda^2)&\bigl(\lal A_0u,BUu\ral+\lal
B^*u,A_1Uu\ral\bigr),\quad\quad
\end{align}
and
\begin{align}
\label{Id1} \lal A_0u,BUu\ral+\lal B^*u,A_1Uu\ral&=
-\lambda\bigl(\|A_1Uu\|^2+\|BUu\|^2-\|\Lambda_0 u\|^2\bigr),\quad
\end{align}
where
\begin{equation}\label{Lam0}
\Lambda_0:=(I+|X|^2)^{1/2}(A_0+BX) (I+|X|^2)^{-1/2}
\end{equation}
is a bounded self-adjoint operator on $\fA_0$.
\end{lemma}

\begin{remark}
\eqref{Id2} represents \cite[identity (2.7)]{AM2012}. The identity
\eqref{Id1} is obtained by combining two remaining identities
from \cite[Lemma 2.2]{AM2012} (see \cite[identities (2.8) and
(2.9)]{AM2012}).
\end{remark}

Below we will only discuss the spectral disposition \eqref{tant}.
Sometimes, we will need its more detail description.

\begin{hypothesis}
\label{HypD} Assume Hypothesis \ref{HypL}. Let $\sigma_0=\spec(A_0)$
and $\sigma_1=\spec(A_1)$. Suppose that the open interval
$\Delta=(\gamma_{\rl},\gamma_{\rr})\subset\bbR$,\,
$\gamma_\rl<\gamma_\rr$,\, serves as a finite gap of the spectral
set $\sigma_1$ and $\sigma_0\subset\Delta$. Let
$d=\dist\bigl(\sigma_0,\sigma_1\bigr)$.
\end{hypothesis}

Now we want to reproduce a known result regarding the position
of the spectrum of the perturbed operator $L=A+V$ and a
known result on the solvability of the associated Riccati equation
\eqref{Ric0} in the spectral case \eqref{tant}. Both these results
have been established in \cite{KMM3} within an approach that is
completely alternative to the methods and technique employed later
on in \cite{MotSel}.

\begin{theorem}  \label{T-KMMa}
Assume Hypothesis \ref{HypD}. Also assume that
\begin{equation}
\label{VdD} \|V\|<\sqrt{d|\Delta|}.
\end{equation}
Then:
\begin{enumerate}
\item[{\rm(i)}]
The spectrum of the block operator matrix $L$ consists of two
isolated parts $\omega_0\subset\Delta$ and
$\omega_1\subset\bbR\setminus\Delta$. Moreover,
\begin{equation}
\label{incl} \min(\omega_0)\geq \gamma_\rl+(d-r_V)\text{\, \, and
\,} \max(\omega_0)\leq\gamma_\rr- (d-r_V),
\end{equation}
where
\begin{equation}
\label{rV}
r_V:=\|V\|\tan\left(\frac{1}{2}\arctan\frac{2\|V\|}{|\Delta|-d}\right)<d.
\end{equation}

\item[{\rm(ii)}] There exists a
unique solution $X\in\cB(\fA_0,\fA_1)$ to the operator Riccati
equation \eqref{Ric0} with the properties
\begin{equation}
\label{sigL} \qquad\qquad\qquad\spec(A_0+BX)=\omega_0\,\text{ \, and
\, }\, \spec(A_1-B^*X^*)=\omega_1.
\end{equation}
Moreover, the spectral subspaces
$\fL_0=\Ran\bigl(\sE_L(\omega_0)\bigr)$ and
$\fL_1=\Ran\bigl(\sE_L(\omega_1)\bigr)$ admit the graph
representations $\fL_0=\cG(X)$ and $\fL_1=\cG(-X^*)$ associated with
the angular operators $X$ and $-X^*$, respectively.

\end{enumerate}
\end{theorem}

\begin{remark}
The formulations in items (i) and (ii) of  Theorem \ref{T-KMMa} are
borrowed from \cite[Theorem 2.4]{AM2012}. Assertion (i) follows from
\cite[Theorem~3.2]{KMM3}. The assertion (ii) is obtained by combining the
statements from \cite[Theorem 2.3]{KMM3} and the existence and
uniqueness results for the Riccati equation \eqref{Ric0} proven in
\cite[Theorem~1\,(i)]{KMM3}.
\end{remark}

\section{The bound on rotation of the spectral subspace \\
and its new proof} \label{Sec3}

The  a priori sharp norm bound for the operator angle between the
unperturbed and perturbed spectral subspaces
$\fA_0=\Ran\bigl(\sE_A(\sigma_0)\bigr)$ and
$\fL_0=\Ran\bigl(\sE_L(\omega_0)\bigr)$, involving not only the
distance $d$ but also another parameter, the length $|\Delta|$ of
the spectral gap $\Delta$, has been established for the first time
in \cite[Theorem 5.3]{MotSel}. This has been done under the
requirement
\begin{equation}
\label{VMS} \| V \| < \sqrt{d(|\Delta| - d)}
\end{equation}
stronger than the condition \eqref{VdD} in Theorem \ref{T-KMMa}.
The main assertion in \cite[Theorem 5.3]{MotSel}, written in
terms of the norm of the difference $\sE_{A}(\sigma_0)-
\sE_{L}(\omega_0)$ of the spectral projections $\sE_{A}(\sigma_0)$
and $\sE_{L}(\omega_0)$, is as follows.

\begin{theorem}[\cite{MotSel}]
\label{T-MS} Assume Hypothesis \ref{HypD}. Assume, in addition, that
inequality \eqref{VMS} holds.  Then
\begin{equation}
\label{thebest}  \|\sE_{A}(\sigma_0)- \sE_{L}(\omega_0)\| \leq
\sin\left(\frac{1}{2}\arctan
\varkappa\bigl(|\Delta|,d,\|V\|\bigr)\right)\quad\left(<\frac{\sqrt{2}}{2}\right),
\end{equation}
where $\omega_0=\spec(L)\cap\Delta$ \,\, and the quantity
$\varkappa(D,d,\slv)$ is defined for
\begin{equation}
\label{Om12} D>0,\quad 0< d\leq \dfrac{D}{2}\quad\text{и}\quad 0\leq
\slv<\sqrt{d(D-d)}
\end{equation}
by
\begin{equation*}
\varkappa(D,d,\slv):=\left\{\begin{array}{cl}
\mbox{\small$\displaystyle\frac{2\slv}{d}$} & \text{if \,} \slv \leq
\mbox{\small$\displaystyle{\frac{1}{2}}$}\sqrt{d\left(D-2d\right)},\\[4mm]
\mbox{\small$\displaystyle\frac{{\slv}D+\sqrt{d(D-d)}\,
\sqrt{(D-2d)^2+4{\slv}^2}}{2\,\bigl(d(D-d)-{\slv}^2\bigr)}$} &
\text{if \,} \slv
>\mbox{\small$\displaystyle{\frac{1}{2}}$}\sqrt{d\left(D-2d\right)}.
\end{array}\right.
\end{equation*}
\end{theorem}

As it was underlined in the Introduction section, our main goal
is to present a proof of Theorem \ref{T-MS} that does not
depend on the approach used in \cite{MotSel} and, in addition, is
simpler then the proof in \cite{MotSel}.

We first prove the bound \eqref{thebest} in a particular case where
the absolute value $|X|$ of the angular operator
$X\in\cB(\fA_0,\fA_1)$ from the graph representation
$\fL_0=\cG(X)$ has an eigenvalue equal to the
norm of $X$. The proof is done by making a straightforward estimate
of this eigenvalue. The general case is easily reduced to the above
particular case by using a quite common limit procedure, involving
orthogonal projections onto the elements of a sequence of expanding
finite-dimensional subspaces in  $\fA_0$ (see, e.g., the proof of
\cite[Theorem~4.1]{AM2012}; cf. the proof of
\cite[Theorem~4.2]{KMM1}). In view of its commonality, we skip this
part of the proof.

\medskip

\noindent \textit{Proof} of Theorem \ref{T-MS}. Throughout the
proof, we assume, without loss of generality, that the interval
$\Delta$ is located on $\bbR$ symmetrically with respect to the
origin, that is, $\gamma_{\rr}=-\gamma_{\rl}=\gamma$. Otherwise, one
may always make the replacement of $A_0$ and $A_1$ respectively by
$A'_0=A_0-c I$ and $A'_1=A_1-c I$ where
$c=(\gamma_{\rl}+\gamma_{\rr})/2$ denotes the center of the interval
$\Delta$. Obviously, such a replacement does not affect the property
of the operator $X$ to be the solution of the transformed equation
\eqref{Ric0}.

The assumptions
$\sigma_0\subset\Delta=(-\gamma,\gamma)$ and
$d=\dist(\sigma_0,\sigma_1)>0$ imply $\sigma_0\subset[-a,a]$ where
$a=\gamma-d$. At the same time
\begin{equation}
\label{A0a} \|A_0\|=a,
\end{equation}
$\gamma=a+d$, and the condition  \eqref{VMS} may be written as
\begin{equation}
\label{VMS1} \|V\|<\sqrt{d(2a+d)}.
\end{equation}

Let $X$ be the very same (unique) solution to \eqref{Ric0} that is spoken
about in the item (ii) of Theorem \ref{T-KMMa}. By our assumptions,
$V\neq 0$ (see Hypothesis \ref{HypL}) and, hence, $X\neq 0$. Suppose
that the absolute value $|X|$ of the operator $X$ has an eigenvalue
$\mu$ coinciding with the norm of $X$, i.e.,
\begin{equation}
\label{muX} \mu=\bigl\||X|\bigr\|=\|X\|>0,
\end{equation}
and let $u$, $\|u\|=1$, be an eigenvector of $|X|$ corresponding to
this eigenvalue, $|X|u=\mu u$. Since $\mu\neq 0$ and
$u=\frac{1}{\mu}|X|u$, one concludes that $u\in\Ran(|X|)$ and, due
to \eqref{Kisom},
\begin{equation}
\label{Uu} \|Uu\|=\|u\|=1,
\end{equation}
where $U$ is the isometry from the polar decomposition $X=U|X|$. By
Lemma \ref{Lcr} we also know that $Uu\in\Dom(A_1)$ and that the
following identities hold:
\begin{align}
\nonumber
\mu\bigl(\|A_1Uu\|^2+&\|BUu\|^2-\|A_0u\|^2-\|B^*u\|^2\bigr)\\
\label{Idl2}
&\,\,=-(1-\mu^2)\bigl(\lal A_0u,BUu\ral+\lal B^*u,A_1Uu\ral\bigr),\\
\label{Idl1} \lal A_0u,BUu\ral+\lal
B^*u,A_1Uu&\ral=-\mu\bigl(\|A_1Uu\|^2+ \|BUu\|^2-\|\Lambda_0
u\|^2\bigr),
\end{align}
where $\Lambda_0$ is a bounded self-adjoint operator on $\fA_0$
being expressed through $A_0$, $B$, and $X$ by~\eqref{Lam0}.

The similarity \eqref{Lam0} implies that
$\spec(\Lambda_0)=\spec(A_0+BX)$ and, hence, by Theorem \ref{T-KMMa}
(ii) we have $\spec(\Lambda_0)=\omega_0$. From the assertion (i) of
the same theorem it then follows that $\|\Lambda_0u\|\leq
\gamma-(d-r_V)<\gamma$, i.e.
\begin{equation}
\label{Lad} \|\Lambda_0u\|< a+d,
\end{equation}
Notice that the spectrum of $A_1$ lies outside the interval
$\Delta=(-a-d,a+d)$. Together with \eqref{Uu} this means
\begin{equation}
\label{Aad} \|A_1Uu\|\geq a+d.
\end{equation}
In view of \eqref{Lad} and \eqref{Aad} we find
\begin{equation*}
\|A_1Uu\|^2+ \|BUu\|^2-\|\Lambda_0 u\|^2>(a+d)^2+
\|BUu\|^2-(a+d)^2\geq 0.
\end{equation*}
Then the relations \eqref{muX} and \eqref{Idl1} require the strict
inequality
\begin{equation}
\label{Help} \lal A_0u,BUu\ral+\lal B^*u,A_1Uu\ral<0.
\end{equation}
On the other hand, because of  \eqref{A0a} and \eqref{VMS1} one has
\begin{align}
\nonumber
\|A_1Uu\|^2+\|BUu\|^2-\|A_0u\|^2-\|B^*u\|^2\geq & (a+d)^2-a^2-\|B\|^2\\
\label{Znam} &=d(2a+d)-{\slv}^2>0,
\end{align}
where, for shortness, we use the notation
\begin{equation*}
{\slv}=\|B\|\quad(=\|V\|).
\end{equation*}
Due to \eqref{Help} and \eqref{Znam}, the identitity \eqref{Idl2}
yields
\begin{equation}
\label{ml1} \mu<1.
\end{equation}
Taking into account \eqref{Znam} and \eqref{ml1}, one can rewrite
\eqref{Idl2} in the form
\begin{equation*}
\frac{\mu}{1-\mu^2}=-\frac{\lal A_0u,BUu\ral+\lal
B^*u,A_1Uu\ral}{\|A_1Uu\|^2+\|BUu\|^2-\|A_0u\|^2-\|B^*u\|^2}\quad(>0)
\end{equation*}
and then conclude that
\begin{equation}
\label{muest}
\frac{\mu}{1-\mu^2}\leq\frac{a\|BUu\|+{\slv}\|A_1Uu\|}{\|A_1Uu\|^2+\|BUu\|^2-a^2-{\slv}^2}.
\end{equation}
Set
\begin{equation*}
x=\|A_1Uu\|\quad\text{and}\quad y=\|BUu\|.
\end{equation*}
In view of \eqref{Aad} we have $x\in[a+d,\infty)$. At the same time
$y\in[0,{\slv}]$. This means that, in any case, the following bound
holds:
\begin{equation}
\label{mufi}
 \frac{\mu}{1-\mu^2}\leq \sup\limits_{(x,y)\in\Omega} \varphi(x,y),
\end{equation}
where $\Omega=[a+d,\infty)\times[0,{\slv}]$ and
\begin{equation}
\label{phi} \varphi(x,y):=\frac{{\slv}x+ay}{x^2+y^2-a^2-{\slv}^2}.
\end{equation}
Elementary calculations show that the largest value of the function
$\varphi$ on the set $\Omega$ is reached at the piece of the
boundary of this set corresponding to $x=a+d$ and $y\in[0,{\slv}]$.
Namely, if $0< {\slv}\leq \sqrt{\frac{1}{2}da}$, then the maximum of
$\varphi$ on $\Omega$ is provided by the point $x=a+d$, $y={\slv}$.
If  $\sqrt{\frac{1}{2}da}< {\slv}<\sqrt{d(2a+d)}$, then the function
$\varphi$ achieves its maximmal value on $\Omega$ at $x=a+d$ and
$$
y=\frac{a[d(2a+d)-{\slv}^2]}{{\slv}(a+d)+a\sqrt{d(2a+d)(a^2+{\slv}^2)}}\,
<{\slv}.
$$
Substitution of the respective maximum point into \eqref{phi}
results in
\begin{equation}
\label{supf} \mbox{\small$\sup\limits_{(x,y)\in\Omega}
\varphi(x,y)=$}\left\{\begin{array}{cl}
\mbox{\small$\dfrac{{\slv}}{d}$} & \text{if \,} \mbox{\small${\slv}\leq \sqrt{\dfrac{1}{2}da}$},\\[4mm]
\mbox{\small$\dfrac{1}{2}\dfrac{{\slv}(a+d)+\sqrt{d(2a+d)}\sqrt{a^2+{\slv}^2}}{d(2a+d)-{\slv}^2}$}
& \text{if \,}\mbox{\small${\slv}>\sqrt{d(2a+d)}$}.
\end{array}
\right.
\end{equation}
Taking into account that
\begin{equation*}
\frac{\mu}{1-\mu^2}=\frac{1}{2}\tan (2\arctan\mu),
\end{equation*}
and recalling that $\mu=\|X\|$, $\slv=\|V\|$, and
$a=\frac{1}{2}|\Delta|-d$, from \eqref{ml1}, \eqref{mufi}, and
\eqref{supf} one derives that
\begin{equation}
\label{Xm} \|X\|\leq \tan\left(\frac{1}{2}\arctan
\varkappa(|\Delta|,d,\|V\|)\right).
\end{equation}
By Theorem \ref{T-KMMa} (ii), the spectral subspace
$\fL_0=\Ran\bigl(\sE_L(\omega_0)\bigr)$ is the graph of the operator
$X$. Hence, the norm of the difference of the orthogonal projections
$\sE_A(\sigma_0)$ and $\sE_L(\omega_0)$ is expressed through $\|X\|$
according to \eqref{EEX} (see, e.g., \cite[Corollary 3.4]{KMM2}). In
view of \eqref{EEX}, unequality \eqref{Xm} is equivalent to the
estimate \eqref{thebest}. Thus, for the case where the absolute
value $|X|$ has an eigenvalue coinciding with $\|X\|$, the proof is
complete.

In the case where $|X|$ does not have an eigenvalue equal to
$\|X\|$, the proof is reduced to the case already considered by
almost literally repeating the reasoning used in \cite{AM2012} to
prove the corresponding estimate for $\|\sE_{A}(\sigma_0)-
\sE_{L}(\omega_0)\|$ when $\sqrt{d(|\Delta| -
d)}\leq\|V\|<\sqrt{d|\Delta|}$ (see \cite[Theorem~4.1]{AM2012}).
Having done this reference to \cite{AM2012}, we consider the whole
proof of Theorem~\ref{T-MS} complete. \qed
\medskip

It only remains to recall (see the proof of \cite[Theorem
2]{MotSel}), that the assertion of Theorem~\ref{Th1} for $\|V\|<d$
is a simple corollary to the more detail estimate~\eqref{thebest}.
For the sake of completeness, we give, nevertheless, a necessary
explanation.

\begin{proof}[Proof of Theorem \ref{Th1} for $\|V\|<d$]
Since $|\Delta| \ge 2 d$ and $\|V\|<d$, inequality
$\|V\|<\sqrt{d(|\Delta|-d)}$ holds automatically. Consequently, by
Theorem \ref{T-MS} one has the bound \eqref{thebest}. Notice that
for fixed values of $\|V\|$ and $d$ satisfying $\|V\|<d$, the
quantity $\varkappa(D,d,\|V\|)$ represents a nonincreasing function
of the variable $D\in[2d,\infty)$. This function acquires its
maximum at $D=2d$, and the maximum equals
\begin{equation}
\nonumber \max\limits_{D: \,D\geq
2d}\varkappa(D,d,\|V\|)=\varkappa(2d,d,\|V\|=\dfrac{2\|V\|d}{d^2-\|V\|^2}=
\tan\Bigl( 2 \arctan \frac{\|V\|}{d} \Bigr).
\end{equation}
Thus, \eqref{thebest} yields \eqref{edif1}.
\end{proof}

\begin{remark}
Optimality of the bounds \eqref{edif1} for $\|V\|<d$ and
\eqref{thebest} for $\|V\|<\sqrt{d(|\Delta|-d)}$ is confirmed by the
respective concrete matrix examples (see \cite[Example~5.5 and
Remark~5.6]{MotSel} and \cite[Example 4.4 and Remark 4.5]{AM2012}).
\end{remark}

This research was supported by the Russian Foundation for Basic
Research and by the De\-u\-t\-sche Forschungsgemeinschaft (DFG). The
support by St.\,Petersburg State University (grant \#
11.38.241.2015) is also acknowledged.


\end{document}